\documentclass[11pt,a4paper]{article}
\usepackage{fullpage}

\usepackage[utf8]{inputenc}
\usepackage{amsthm,amssymb}
\usepackage{amsmath}

\usepackage{authblk}
\usepackage{comment}
\usepackage{algorithm}
\usepackage{algpseudocode}
\usepackage{tcolorbox}
\usepackage{mathtools,thmtools}
\usepackage[mathcal,mathscr]{eucal}
\usepackage{epsfig,graphicx,graphics,color}
\usepackage{caption}
\usepackage{subcaption}
\usepackage{enumerate}
\usepackage{enumitem, crossreftools}
\usepackage[sort]{cite}
\usepackage{titling}
\usepackage{algorithm}
\usepackage{tikz}
\usepackage{hyperref}
\usepackage{mathtools}
\usepackage{makecell}
\usepackage{hhline}
\usepackage{color, colortbl}
\definecolor{Gray}{gray}{0.9}

\hypersetup{
    colorlinks=true,       % false: boxed links; true: colored links
    linkcolor=blue,        % color of internal links (change box color with linkbordercolor)
    citecolor=red,         % color of links to bibliography
    filecolor=magenta,     % color of file links
    urlcolor=cyan,         % color of external links
    linktocpage=true
}

\newcounter{algsubstate}


\theoremstyle{plain}
\newtheorem{theorem}{Theorem}
\newtheorem{lemma}[theorem]{Lemma}

\newtheorem{claim}[theorem]{Claim}

\newtheorem{qu}{Question}

\theoremstyle{definition}

\newtheorem{ex}[theorem]{Example}

\newtheorem{remark}[theorem]{Remark}

\newcommand{\cB}{\mathcal{B}}
\newcommand{\cC}{\mathcal{C}}
\newcommand{\cD}{\mathcal{D}}

\newcommand{\cF}{\mathcal{F}}
\newcommand{\cG}{\mathcal{G}}
\newcommand{\cI}{\mathcal{I}}

\newcommand{\cH}{\mathcal{H}}
\newcommand{\cK}{\mathcal{K}}
\newcommand{\cM}{\mathcal{M}}

\newcommand{\cT}{\mathcal{T}}

\newcommand{\bM}{\mathbb{M}}

\newcommand{\FC}{\mathbb{T}}
\newcommand{\IS}{\mathbb{I}}

\def\bx{\mathbf{x}}

\DeclareMathOperator*{\cl}{cl}

\def\final{0}  % set this to 1 to get a comment-free version
\ifnum\final=0  %namely if we allow comments in the output
\newcommand{\krnote}[1]{{\color{red}[{\tiny Krist\'of: \bf #1}]\marginpar{\color{red}*}}}
\newcommand{\enote}[1]{{\color{blue}[{\tiny Endre: \bf #1}]\marginpar{\color{Blue}*}}}
\newcommand{\knote}[1]{{\color{orange}[{\tiny Kaz: \bf #1}]\marginpar{\color{red}*}}}
\else % in this case [final=1] we don't want any comments to show
\newcommand{\krnote}[1]{}
\newcommand{\enote}[1]{}
\newcommand{\knote}[1]{}
\fi

\title{Matroid Horn Functions}

\author{
	Krist{\'o}f B{\'e}rczi\thanks{MTA-ELTE Momentum Matroid Optimization Research Group and MTA-ELTE Egerv\'ary Research Group, Department of Operations Research, E{\"o}tv{\"o}s Lor{\'a}nd University, Budapest, Hungary. 
	Email: \texttt{kristof.berczi@ttk.elte.hu}.} 
\and
	Endre Boros\thanks{MSIS Department and RUTCOR, Rutgers University, Piscataway, New Jersey, USA. Email: \texttt{endre.boros@rutgers.edu}.}
\and
	Kazuhisa Makino\thanks{Research Institute for Mathematical Sciences (RIMS) Kyoto University, Kyoto, Japan. Email: \texttt{makino@kurims.kyoto.ac.jp}.}
}

\begin{document}
\date{}
\maketitle

\begin{abstract}
Hypergraph Horn functions were introduced as a subclass of Horn functions that can be represented by a collection of circular implication rules. These functions possess distinguished structural and computational properties. In particular, their characterizations in terms of implicate-duality and the closure operator provide extensions of matroid duality and the Mac Lane\,--\,Steinitz exchange property of matroid closure, respectively. 

In the present paper, we introduce a subclass of hypergraph Horn functions that we call \emph{matroid Horn} functions. We provide multiple characterizations of matroid Horn functions in terms of their canonical and complete CNF representations. We also study the Boolean minimization problem for this class, where the goal is to find a minimum size representation of a matroid Horn function given by a CNF representation. While there are various ways to measure the size of a CNF, we focus on the \emph{number of circuits} and \emph{circuit clauses}. We determine the size of an optimal representation for binary matroids, and give lower and upper bounds in the uniform case. For uniform matroids, we show a strong connection between our problem and Tur\'an systems that might be of independent combinatorial interest.

  \bigskip

\end{abstract}

%%%%%%%%%%%%%%%%
\section{Introduction}
\label{sec:introduction}
%%%%%%%%%%%%%%%%

Hypergraph Horn functions were introduced in \cite{hypergraph_horn} as a special subclass of Horn functions that is rich in mathematical structure and has advantageous algorithmic properties. The family of hypergraph Horn functions turned out to be highly structured yet general enough to provide insight into the structure of Horn functions, e.g., they generalize equivalence relations, and more generally matroids. The authors introduced the notion of implicate-duality or $i$-duality, and showed that every Boolean function has a unique $i$-dual, which is always hypergraph Horn. In particular, $i$-duality generalizes matroid duality in the sense that if a hypergraph Horn function corresponds to a matroid, then its $i$-dual corresponds to the dual matroid.

In this paper, we consider hypergraph Horn functions for which the underlying hypergraph corresponds to the circuits of a matroid, and call them \emph{matroid Horn} functions. These functions possess many interesting properties. Some are functional properties that are related only to the structure of true and false sets of the function, while others are related to the structure of possible representations of the function. We provide characterizations of matroid Horn functions in terms of prime implicates, the closure operator and the rank function, and the $i$-dual. Meanwhile, we also explain the correspondence between the basic notions of matroid theory and matroid Horn logic. 

One of the most intriguing questions related to Boolean functions is the so called \emph{Boolean minimization problem} (BM): given a representation of a Boolean function in a conjunctive normal form (CNF), find an equivalent CNF representation of minimum size with respect to a certain objective. There are different ways to measure the size of a CNF, standard examples being the number of clauses and the total number of literals. Finding a short representation is motivated by many applications, for example, such a representation can be used to reduce the size of the knowledge base in a propositional expert system, which in turn improves the performance of the system. An unsatisfiable formula can be easily recognized from its shortest CNF representation for both above measures of the output size. Therefore BM contains the CNF satisfiability problem (SAT) as a special case, implying that it is NP-hard. However, while SAT is NP-complete (i.e. $\Sigma^p_1$-complete~\cite{C71}), BM is $\Sigma^p_2$-complete~\cite{U01,UVS06}, suggesting that BM is probably more difficult than SAT. We consider the problem of finding short representations of matroid Horn functions. We also discuss minimum representations of matroids in terms of circuit generation. We show that for binary matroids, the set of chordless cycles provides the unique optimal representation for all objectives considered. For uniform matroids, we prove that our problem is closely related to the minimum size of Tur\'an systems and provide bounds as a function of Tur\'an numbers; this result might be of independent combinatorial interest. 
\medskip

The rest of the paper is organized as follows. Basic definitions and results on hypergraphs, Horn functions and matroids are presented in Section~\ref{sec:preliminaries} as in~\cite{hypergraph_horn}. In Section~\ref{sec:matroid}, we introduce the class of matroid Horn functions and provide several equivalent characterizations. The Boolean minimization problem for matroid Horn functions is discussed in Section~\ref{sec:representation}. Finally, we close the paper by some open problems in Section~\ref{sec:conclusions}.

%%%%%%%%%%%%%%%%
\section{Preliminaries}
\label{sec:preliminaries}
%%%%%%%%%%%%%%%%

For discussing the result of the paper, we need basic results on hypergraphs, Horn functions and matroids. To make the paper self-contained, we repeat the definitions as summarized in~\cite{hypergraph_horn}.  

\paragraph{Basic notation.}

For a positive integer $k$, we use $[k]=\{1,\dots,k\}$. Given a ground set $E$ together with subsets $X,Y\subseteq E$, the \emph{difference} of $X$ and $Y$ is denoted by $X\setminus Y$. If $Y$ consists of a single element $y$, then $X\setminus \{y\}$ and $X\cup\{y\}$ are abbreviated as $X-y$ and $X+y$, respectively. The symmetric difference of $X$ and $Y$ is denoted by $X\triangle Y= (X\setminus Y)\cup(Y\setminus X)$.

\paragraph{Hypergraphs.}

For a finite set $V$, a family $\cH\subseteq 2^V$ of its subsets is called  a {\em hypergraph}, where 
$H \in \cH$ is  called a {\em hyperedge} of $\cH$. A hypergraph $\cH$ is called {\em Sperner} if no distinct hyperedges of $\cH$ contain one another. For a subset $X\subseteq V$ we call $V\setminus X$ its \emph{complement}, and we denote by $\cH^c=\{V\setminus H\mid H\in\cH\}$ the \emph{complementary family} of $\cH$. Note that the operator ``$c$" is an  involution, i.e., $\left(\cH^c\right)^c=\cH$ holds for all hypergraphs. A subset $T\subseteq V$ is called a \emph{transversal} of $\cH$ if $T\cap H\neq\emptyset$ for all hyperedges $H\in\cH$. We denote by $\cH^d$ the \emph{family of minimal transversals} of $\cH$. For simplicity, we omit the parentheses from the notation when applying these operations repeatedly, for example, we write $\cH^{dc}=(\cH^d)^c$ and $\cH^{cd}=(\cH^c)^d$. For a Sperner hypergraph $\cH$, the operator ``$d$'' is also an involution, hence we have $\cH^{dccd}=\cH^{cddc}=\cH$ for such hypergraphs. Note that the family $\cH^{dc}$ is the \emph{family of maximal independent sets} of $\cH$, where a set $I \subseteq V$ is called {\em independent}  of $\cH$ if it contains no hyperedge of $\cH$. The family $\cH^{cd}$ consists of all minimal subsets of $V$ that are not contained in a  hyperedge of $\cH$. Every Sperner hypergraph $\cH$ defines an \emph{independence system} $\cI$ consisting of all independent sets of $\cH$~\cite{Welsh1976}. In this correspondence,  maximal independent sets of $\cH$ are also called \emph{bases} of the associated independence system $\cI$, while minimal {\em dependent} sets (i.e.,  hyperedges) of $\cH$ are called  \emph{circuits} of $\cI$. Typically we reserve the notation $\cB$ and $\cC$ for the families of bases and circuits, respectively. Note that for all independence systems we have 
\begin{equation}\label{e-bases-circuits}
	\cB^{cd}=\cC ~~~\text{ and equivalently }~~~ \cC^{dc}=\cB.
\end{equation}
Independent sets are exactly the subsets of hyperedges of $\cB=\cC^{dc}$, or equivalently, subsets that do not contain a hyperedge of $\cC$. 

For an arbitrary hypergraph $\cH\subseteq 2^V$ we denote by $\cH^\cap$ the \emph{intersection closure} of $\cH$, defined by 
\[
\cH^\cap  =  \left\{ \left.\bigcap_{F\in\cF} F ~\right|~ \cF\subseteq \cH \right\}.
\]
We note that the intersection of an empty family is defined as $V$, and thus we have $V\in \cH^\cap$ for all hypergraphs $\cH\subseteq 2^V$. Analogously, we denote by $\cH^\cup$ the \emph{union closure} of $\cH$ defined as
\[
\cH^\cup  =  \left\{ \left.\bigcup_{F\in\cF} F ~\right|~ \cF\subseteq \cH \right\}.
\]
We note that the union of an empty family is defined as the empty set, and thus we have $\emptyset\in\cH^\cup$ for all hypergraphs $\cH$. 

\paragraph{Hypergraph Horn functions.}

We denote by $V$ the set of $n$ Boolean variables $x$ and call these together with their negations $\overline{x}$ as \emph{literals}. Members of $V$ are called \emph{positive literals} while their negations are \emph{negative literals}. A disjunction of a subset of the literals is called a \emph{clause} if it contains no complementary pair of literals $x$ and $\overline{x}$, and the conjunction of clauses is a called a \emph{conjunctive normal form} (or in short a {\em CNF}). It is well-known (see e.g.,~\cite{CH11}) that every Boolean function $f:\{0,1\}^V\rightarrow\{0,1\}$ can be represented by a CNF, typically not in a unique way. We use Greek letters to denote CNFs, and Latin letters to denote Boolean functions.

Truth assignments (i.e., Boolean vectors) $\bx=(x_1,\dots,x_n)\in\{0,1\}^V$ can be viewed equivalently as characteristic vectors of subsets. For a subset $Z\subseteq V$ we denote by $\chi_Z\in \{0,1\}^V$ its characteristic vector, i.e., $(\chi_Z)_i=1$ if and only if $i \in Z$. Since we use primarily a combinatorial notation in this paper, we use for a function $f$ and Boolean expression $\Phi$ the notation $f(Z)$ and $\Phi(Z)$ instead of $f(\chi_Z)$ and $\Phi(\chi_Z)$, to denote the evaluation of $f$ and  $\Phi$ at the binary vector $\chi_Z$, respectively. We say that a set $Z\subseteq V$ is a \emph{true set} of $f$ if $f(Z)=1$, and a \emph{false set} otherwise. We denote by $\cT(f)$ and $\cF(f)$ the \emph{families of true sets} and \emph{false sets} of $f$, respectively. Every Boolean expression defines/represents a unique Boolean function. If $A$ and $B$ denote Boolean functions or Boolean expressions, we write $A=B$ if $A(X)=B(X)$ for all $X\subseteq V$, and  $A\leq B$ if $A(X)\leq B(X)$ for all $X\subseteq V$, where $B$ is called a \emph{majorant} of $A$ in the latter case. We also write $A <B$ if $A\leq B$ and $A\not=B$. 

A clause is called \emph{definite Horn} if it contains exactly one positive literal. It is easy to see that definite Horn clauses represent simple implications. Namely, for a proper subset $B\subsetneq V$ and $v\in V\setminus B$  the implication $B\to v$ is equivalent to the definite Horn clause $C=v\vee \left(\bigvee_{u\in B}\overline{u}\right)$: the true sets of both expressions are exactly the sets $T\subseteq V$ such that either $T\not\supseteq B$ or $T\supseteq  B\cup\{v\}$. We call $B$ the \emph{body} of the Horn clause. A CNF is called  {\em definite Horn} if it consists of definite Horn clauses, and a Boolean function is called  {\em definite Horn} if it can be represented by a definite Horn CNF. The following characterization of definite Horn functions is  well-known. 

\begin{lemma}[see e.g., \cite{McKinsey1943,CH11}]
\label{l-E-intersection-closed}
	A Boolean function $f$ is definite Horn if and only if the family $\cT(f)$  of true sets of $f$ is closed under intersection and contains $V$.
	\hfill$\Box$
\end{lemma}

Lemma~\ref{l-E-intersection-closed} implies that for any  set $Z\subseteq V$ there exists a unique minimal true set containing $Z$, the so-called \emph{closure} of $Z$ denoted by $\FC_h(Z)$. In fact such a closure can be computed efficiently from  any definite Horn CNF representation $\Phi$ of $h$ by the so-called \emph{forward chaining procedure} (see e.g.,~\cite{CH11}): Let $A\subseteq V$ denote the set of all variables $v\in V \setminus Z$ for which there exists a clause $B\to v$ in $\Phi$ with $B\subseteq Z$ and $v \in V \setminus Z$, and define $\FC^1_{\Phi}(Z)= Z\cup A$. For $i\geq 2$ we define $\FC^i_{\Phi}(Z)=\FC^1_{\Phi}(\FC^{i-1}_{\Phi}(Z))$. Since $ Z\subseteq \FC^{1}_{\Phi}(Z) \subseteq  \FC^{2}_{\Phi}(Z) \subseteq \dots \subseteq V$, $\FC^{i+1}_{\Phi}(Z)=\FC^i_{\Phi}(Z)$ holds for some integer $i\leq n$. Let $i^*$ be the smallest such index $i$. Then we have $\FC^t_{\Phi}(Z)=\FC^{i^*}_{\Phi}(Z)$ for all $t\geq i^*$, and $\FC^{i^*}_{\Phi}(Z)$ is the minimal true set of $\Phi$ that contains $Z$. Thus we can define $\FC_{\Phi}(Z)=\FC^{i^*}_{\Phi}(Z)$, and we say that  {\em $Z$ can be closed by $\Phi$ in $i^*$ steps}. While we may have $\FC^1_{\Phi}(Z)\neq \FC^1_{\Psi}(Z)$ for different definite Horn CNFs $\Phi$ and $\Psi$ representing the same function $h$, it can be shown (see e.g.,~\cite{CH11}) that the resulting set $\FC_\Phi(Z)$ does not depend on the particular choice of the representation $\Phi$ of $h$, but only on the underlying function $h$, that is, we can write $\FC_h(Z)=\FC_\Phi(Z)$ for this uniquely defined closure of $Z$.  Note that $\FC$ is in fact a closure operator in finite set theory,  and hence we call a subset $Z\subseteq V$ \emph{closed} (with respect to $h$) if $\FC_h(Z)=Z$. It is not difficult to check that a set is closed with respect to $h$ if and only if it is a true set of $h$. 

For a definite Horn function $h$, a clause $B\to v$ is called an \emph{implicate} of $h$ if it is a majorant of $h$, that is,  if $(B\to v) \geq h$. An implicate $B\to v$ of $h$ is \emph{prime} if $h$ has no other implicate $B'\rightarrow v$ with $B'\subsetneq B$. The following lemma characterizes implicates of definite Horn functions in terms of the closure operator.  

\begin{lemma}[see e.g.,~\cite{CH11}]\label{l-E-implicate}
A clause $B\to v$ is an implicate of a definite Horn function $h$ if and only if $v\in \FC_h(B)\setminus B$. 
\hfill$\Box$
\end{lemma}

Given a definite Horn CNF $\Phi$, a subset $A\subseteq V$ and a variable $v\in V\setminus A$, we write $A\overset{\Phi}{\to}v$ to indicate that $A\to v$ is an implicate of $\Phi$, or equivalently that $v\in \FC_{\Phi}(A)$. To indicate the opposite, that is that $A\to v$ is not an implicate of $\Phi$, or equivalently, that $v\not\in \FC_{\Phi}(A)$ we may simply write $A\overset{\Phi}{\nrightarrow}v$.

A subset $K\subseteq V$ is called a \emph{key} of the definite Horn function $h$ if $\FC_h(K)=V$. We denote by 
\begin{equation*}
\cK(h)=\{K\subseteq V\mid \FC_h(K)=V\ \text{and}\ \FC_h(K')\neq V\ \text{for all $K'\subsetneq K$}\},
\end{equation*}
the family of \emph{minimal keys} of $h$, and $\cK(h)$ is called  the {\em key set} of $h$. A true set $T\in\cT(h)$ is called \emph{nontrivial} if $T\neq V$. We denote by 
\begin{equation*}
	\cM(h)=\{T\subsetneq V\mid h(T)=1\ \text{and}\ h(T')=0\ \text{for all $T\subsetneq T'\subsetneq V$}\},
\end{equation*}
the family of \emph{maximal nontrivial true sets} of $h$. Note that $\cM(h)$ is a subfamily of the so-called \emph{characteristic models} of $h$~\cite{KhardonRoth1996}.

Both families, $\cK(h)$ and $\cM(h)$ are Sperner hypergraphs for all definite Horn functions $h$. It is not difficult to verify that maximal nontrivial true sets form the family of maximal independent sets of the family of minimal keys, and minimal keys are exactly the minimal sets that are not contained in a maximal nontrivial true set.

\begin{lemma}[see e.g.,~\cite{CH11}]\label{l-E-M(h)-K(h)}
For a definite Horn function $h$ we have $\cM(h)=\cK(h)^{dc}$ and $\cK(h)=\cM(h)^{cd}$.
\hfill$\Box$
\end{lemma}

Given a Boolean function $f:2^V\to \{0,1\}$, we call an implicate $A\to v$ of it \emph{circular} if $((A+v)-u)\to u$ is also an implicate of $f$ for every $u\in A$. We say that a subset $I\subseteq V$ is an \emph{implicate set} of $f$ if $(I-v) \to v$ is an implicate of $f$ for all $v\in I$, and denote the \emph{family of implicate sets} by $\cI(f)$. By definition, we have $\emptyset\in\cI(f)$ for all Boolean functions $f$. Furthermore, let us observe that the hypergraph $\cI(f)$ is union closed, that is, $I,J\in\cI(f)$ implies $I\cup J\in\cI(f)$.

To a hypergraph $\cH\subseteq 2^V$ we associate the definite Horn CNF $\Phi_\cH$ defined as
\begin{equation*}\label{e-circulant}
	\Phi_\cH  =  \bigwedge_{H\in\cH} \left(\bigwedge_{v\in H} \left((H-v)\to v\right)\right)
\end{equation*}
and call $\Phi_{\cH}$ the \emph{circular CNF} associated to the hypergraph $\cH$. We say that a definite Horn function $h:2^V\to \{0,1\}$ is \emph{hypergraph Horn} if it has a circular CNF representation, that is, if there exists a hypergraph $\cH\subseteq 2^V$ such that $h  =  \Phi_\cH$. Hypergraph Horn functions were introduced in \cite{hypergraph_horn} as highly structured objects with distinguished algorithmic properties. In particular, the following results appeared in \cite[Corollary 10, Theorem 11]{hypergraph_horn}. 

\begin{theorem}\label{thm:summary}
\mbox{}
\begin{enumerate}[label={\rm (\alph*)}]\itemsep0em
\item \label{it:e} A definite Horn function $h$ is hypergraph Horn if and only if for every false set  $F$ of $h$, there exists an implicate set $I\in\cG(h)$ such that $|I\setminus F|=1$.
\item \label{e2-true-set-of-h} For a hypergraph Horn function $h$, we have $\cT(h) = \{ T\subseteq V \mid    \nexists ~ I\in\cI(h) \text{ with } ~ |I\setminus T|=1  \}$.
\end{enumerate}
\end{theorem}

The conjunction of two hypergraph Horn functions $h_1$ and $h_2$ is also hypergraph Horn, since if $\cH_1$ and $\cH_2$ are hypergraph such that $h_i=\Phi_{\cH_i}$ for $i=1,2$, then $h_1\wedge h_2=\Phi_{\cH_1\cup\cH_2}$. Therefore any Boolean function $f$ admits a unique minimal hypergraph Horn majorant which we denote by $f^{\circ}=\Phi_{\cI(f)}$. The \emph{implicate-dual} $f^i$ of a Boolean function $f$ is the unique function satisfying the equality
\begin{equation}\label{e-f->f^i}
\cT(f^i)  =  \cI(f)^c.
\end{equation}
Since $\cI(f)=\cI(f^\circ)$ for every Boolean function $f$, we have $f^i=(f^\circ)^i$. It can be further showed that $f^i$ is hypergraph Horn for every Boolean function $f$.

\paragraph{Matroids.} 

We give a brief introduction into matroid theory, and refer the reader to~\cite{oxley2011matroid} for further details. A \emph{matroid} $\bM=(E,\cC)$ is defined by its \emph{ground set} $E$ and its \emph{family of circuits} $\cC\subseteq 2^E$ that satisfies the following \emph{circuit axioms}. 
\begin{enumerate}[label={(C\arabic*)}]\itemsep0em
\item \label{cax1} $\emptyset\notin\cC$.
\item \label{cax2} If $C_1,C_2\in\cC$, then $C_1\not\subset C_2$.
\item \label{cax3} If $C_1,C_2\in\cC$ are distinct and $u\in C_1\cap C_2$, then there exists $C_3\in\cC$ such that $C_3\subseteq (C_1\cup C_2) - u$.
\end{enumerate}

For a family $\cC \subseteq 2^V$ that satisfies satisfies the circuit axioms, i.e. it is the family of circuits of some matroid $\bM$, let $\cI$ denote the family of independent sets of $\cC$. Then it is known that it satisfies so-called \emph{independence axioms}: (I1) $\emptyset\in\cI$, (I2) $X\subseteq Y\in \cI\Rightarrow X\in\cI$, and (I3) $X,Y\in\cI,\ |X|<|Y|\Rightarrow\exists e\in Y-X\ \text{s.t.}\ X+e\in\cI$. The \emph{rank} of a set $X\subseteq V$ is the maximum size of an independent subset of $X$ and is denoted by ${\it rank}_\bM(X)$. The maximal independent sets of $M$ are called \emph{bases}.  A set $X\subseteq V$ is called \emph{closed} (or a \emph{flat} or a \emph{subspace}) if ${\it rank}_{\bM}(X+v)>{\it rank}_{\bM}(X)$ for every $v\in V\setminus X$. The \emph{closure} (or \emph{span}) of a set $X\subseteq V$ is defined as $\cl_{\bM}(X)=\{v\in V\mid {\it rank}_{\bM}(X+v)={\it rank}_{\bM}(X)\}$. It is known that $X$ is closed if and only if $\cl_{\bM}(X)=X$. A \emph{hyperplane} is a closed set of rank ${\it rank}_{\bM}(V)-1$. 

Like circuit and independence axioms, each concept has axioms for matroids. For example, the closure operator of a matroid $\bM$ satisfies the following \emph{closure axioms}.
\begin{enumerate}[label=(CL\arabic*)]\itemsep0em
\item \label{clax1} $X\subseteq \cl_{\bM} (X)$ for all $X \subseteq V$.
\item \label{clax2} $\cl_{\bM}(X)= \cl_{\bM}(\cl_{\bM} (X))$ for all $X \subseteq V$.
\item \label{clax3} $\cl_{\bM}(X) \subseteq  \cl_{\bM}(Y)$ for all $X, Y \subseteq V$ with $X \subseteq Y$. 
\item \label{clax4} $v \in  \cl_{\bM} (X  + u ) \setminus \cl_{\bM} (X)$ implies that $u \in  \cl_{\bM} (X  + v ) \setminus \cl_{\bM} (X)$ for all  $u,v \in V$ and all  $X \subseteq V$. 
\end{enumerate}

It is known that general closure operators are defined as the ones that satisfy the axioms \ref{clax1}, \ref{clax2} and \ref{clax3}, and are equivalent to operators  $\FC_h$ for definite Horn functions $f$.  Axiom \ref{clax4} is called {\em Mac Lane–Steinitz exchange property} \cite{Welsh1976}. As we have seen in Theorem~\ref{thm:summary}\ref{it:e}, it is generalized to closure operators of hypergraph Horn functions.  

%%%%%%%%%%%%%%%%
\section{Matroid Horn Functions}
\label{sec:matroid}
%%%%%%%%%%%%%%%%

We consider hypergraph Horn functions associated to families of circuits of a matroid, and study their properties and relations to matroid theory. We call a definite Horn function \emph{matroidal} or \emph{matroid Horn} if $h=\Phi_{\cC}$ for the family of circuits $\cC$ of a matroid $\bM=(V,\cC)$. 

For a matroid Horn function $h$, let us call its CNF representation $h=\Phi_\cC$ \emph{canonical} if the hypergraph $\cC$ satisfies the circuit axioms \ref{cax1}--\ref{cax3}, i.e., $\cC$ is the family of circuits of some matroid. Our first main result provides characterizations of matroid Horn functions in terms of their canonical representations.

\begin{theorem}\label{t-matroidal-formulaic-equivalences}
Let $\cC\subseteq 2^V$ be a nontrivial Sperner hypergraph and let $h$ be the hypergraph Horn function represented by $\Phi_{\cC}$. Then the following are equivalent.
\begin{enumerate}[label={\rm (\roman*)}]\itemsep0em
\item \label{fo-equiv:i}  $\cC$ satisfies the circuit axiom \ref{cax3}.
\item \label{fo-equiv:ii}  $\cK(h)=\cC^{dc}$.
\item \label{fo-equiv:iii}  $\cM(h)=\cC^{dcdc}$.
\item \label{fo-equiv:iv}  $\cT(h)=\left(\cC^{dcdc}\right)^{\cap}$. 
%		\item[(v)] $\cT(h)\supseteq \cC^{dcdc}$
\end{enumerate}
\end{theorem}

Let us remark that claim \ref{fo-equiv:ii} in the above theorem means that the minimal keys of a matroid Horn function are exactly the bases of the corresponding matroid by \eqref{e-bases-circuits}. Furthermore, claims \ref{fo-equiv:iii} and \ref{fo-equiv:iv} imply that the set of characteristic models of a matroid Horn function are exactly its maximal nontrivial true sets. Note that while $\cC^{dcdc}$ can be the characteristic set of $\Phi_\cC$ only for matroid Horn functions, $\cM(h)$ may be the characteristic set of a non-matroidal hypergraph Horn function, too, see Example \ref{e-nonmatroidal-M-char}. 

Let us recall that every Boolean function $f$ has a unique set of prime implicates, and the unique CNF representation that contains all prime implicates of $f$ is called the \emph{complete CNF} of $f$, see e.g. \cite{CH11}. Our next main result provides characterizations of matroid Horn functions in terms of their complete CNF.

\begin{theorem}\label{t-matroidal-functional-equivalences}
For a definite Horn function $h$, the following are equivalent. 
\begin{enumerate}[label={\rm (\roman*)}]\itemsep0em
\item \label{fu-equiv:i} The function $h$ is matroid Horn.
\item \label{fu-equiv:ii} The complete CNF of $h$ is circular.
\item \label{fu-equiv:iii} The implicate-dual function $h^i$ is matroid Horn. 
\item \label{fu-equiv:iv} The complete CNF of $h^i$ is circular. 
\end{enumerate}
\end{theorem}

For the proof of the theorems, we need a number of technical lemmas providing additional characterizations of matroid Horn functions that might be of interest on their own. In fact, we will show that if $h=\Phi_\cC$ for the family of circuits $\cC$ of a matroid $\bM$, then the implicate-dual $h^i$ corresponds to the so called {\em dual matroid} $\bM^*$ the bases of which are the complements of the bases of $\bM$. Thus, we have $h^i=\Phi_{\cC^{dcd}}$, where the family of circuits $\cC^{dcd}$ of the dual matroid $\bM^*$ is also known as the family of \emph{cuts} of matroid $\bM$. 

Two closure operators $\cl_{\bM}$ and $\FC_{h}$ was introduced both for matroids $\bM$ and for definite Horn functions $h$. The next lemma shows that they coincide if they are from circuits of matroids. 

\begin{lemma}\label{lem:cl}
Let $\cC$ be the family of circuits of a matroid $\bM$, and let $h$ be a matroid Horn function represented by $\Phi_\cC$. Then we have $\FC_{h}(X)=\cl_{\bM}(X)$ for all $X\subseteq V$.
\end{lemma}
\begin{proof}
By definition, $\FC_{h}(X)$ is the smallest true set of $h$ containing $X$. Observe that $\cl_{\bM}(X)$ is a true set of $h$ containing $X$. Indeed, if $\cl_{\bM}(X)$ is not a true set, then there exists a set $C\in\cC$ such that $|C\setminus\cl_{\bM}(X)|=1$. But then ${\it rank}_{\bM}(\cl_{\bM}(X) \cup  C)= {\it rank}_{\bM}(\cl_{\bM}(X))$, contradicting $\cl_{\bM}(X)$ being closed. This shows that $\FC_{h}(X)\subseteq\cl_{\bM}(X)$.

To see the other direction, we show that $\FC_{h}(X)$ is a closed set in $\bM$. If this does not hold, then there is an element $v\in V\setminus \FC_{h}(X)$ such that ${\it rank}_{\bM}(\FC_{h}(X)+v)={\it rank}_{\bM}(\FC_{h}(X))$. That is, $\FC_{h}(X)+v$ contains a circuit $C\in\cC$ with $|C\setminus \FC_{h}(X)|=1$, which implies that $\FC_{h}(X)$ is not a true set of $h$, a contradiction.  Thus we get $\cl_{\bM}(X)\subseteq \FC_{h}(X)$.  
\end{proof}

Once we understand Lemma \ref{lem:cl}, we can see the following relationship between matorids and matroid Horn functions in Table \ref{tab:table1}. 

\begin{table}[htb]
    \centering
    \renewcommand*{\arraystretch}{1.5}
    \begin{tabular}{|p{0.36\textwidth}|p{0.36\textwidth}|}
    \hline
    \rowcolor{Gray}
    \makecell{\textbf{Matroid $\bM$ with} \\ \textbf{circuit family $\cC$}} & \makecell{\textbf{Matroid Horn function $h$}\\\textbf{represented by $\Phi_\cC$}}  \\
    \hhline{|=|=|}
         \makecell{Bases of $\bM$} & \makecell{Minimal keys of $h$} \\
         \hline
          \makecell{Hyperplanes of $\bM$} & \makecell{Maximal nontrivial true sets of $h$}\\
         \hline
          \makecell{Closed sets of $\bM$} & \makecell{True sets of $h$}\\
          \hline
    \end{tabular}
    \caption{Correspondence between matroids $\bM$ and matroid Horn functions $h$ defined from the circuits $\cC$ of $\bM$ as in Lemma \ref{lem:cl}.}
    \label{tab:table1}
\end{table}

These observations together with matroid theory show that \ref{fo-equiv:ii}--\ref{fo-equiv:iv} of Theorem \ref{t-matroidal-formulaic-equivalences} follows from \ref{fo-equiv:i}. 

\begin{lemma}\label{lemma-matroid-formul-part1}
Let $\cC\subseteq 2^V$ be a hypergraph satisfying the circuit axioms \ref{cax1}--\ref{cax3}, and let $h$ be the hypergraph Horn function represented by $\Phi_{\cC}$. Then the following hold. 
	\begin{enumerate}[label={\rm (\alph*)}]\itemsep0em
		\item \label{p1:a} $\cK(h)=\cC^{dc}$.
		\item \label{p1:b} $\cM(h)=\cC^{dcdc}$.
		\item \label{p1:c} $\cT(h)=\left(\cC^{dcdc}\right)^{\cap}$. 
%		\item[(v)] $\cT(h)\supseteq \cC^{dcdc}$. 
	\end{enumerate}
\end{lemma}
\begin{proof}
It is known that $\cC^{dc}$ and $\cC^{dcdc}$ are families of bases and hyperplanes of matroid $\cM$,  respectively, see e.g. \cite{Welsh1976}, proving \ref{p1:a} and \ref{p1:b}. Recall that any closed set of the matroid can be obtained as the intersection of hyperplanes, which immediately proves \ref{p1:c}.
\end{proof}

By Lemma~\ref{lemma-matroid-formul-part1},  we can say that the converse of the lemma is a nontrivial part of Theorem \ref{t-matroidal-formulaic-equivalences}. Toward showing it, we next provide characterizations of matroid Horn functions in terms of the rank function, the closure, and the core, which are interesting on their own. 

We associate a rank function to an arbitrary hypergraph $\cH\subseteq 2^V$. For a subset $X\subseteq V$ its \emph{rank} is defined as 
\begin{equation}\label{e-rank}
	r_{\cH}(X)  =  \max_{B\in\cH^{dc}} |B\cap X|.
\end{equation}
Note that this coincides with the standard matroid rank function, whenever $\cH$ is the family of circuits of a matroid, see e.g., \cite{Welsh1976}. 

\begin{lemma}\label{t-new-core}
Let $\cC\subseteq 2^V$ be a nontrivial Sperner hypergraph, and let $h$ be the hypergraph Horn function represented by $\Phi_\cC$.  Then the following are equivalent:
\begin{enumerate}[label={\rm (\roman*)}]\itemsep0em
	\item \label{core:i} $\cC$ satisfies the circuit axiom \ref{cax3}.
	\item \label{core:ii} For any $X\subseteq V$ and any $v \in V \setminus X$, we have 
	$r_\cC(X+v)=r_\cC(X)$ if and only if there exists $C\in\cC$ with $C\setminus X =\{v\}$.
	\item \label{core:iii} For any $X\subseteq V$,  we have $r_\cC(\FC_h(X)) = r_\cC(X)$.
    \item \label{core:iv} For any $X\subseteq V$, we have $\FC_h(X) = \{ v\in V\mid r_\cC(X+v)=r_\cC(X)$.
	\item \label{core:v} For any   $X\subseteq V$,  we have $\FC_h(X) = \FC^1_{\Phi_\cC}(X)$.
	\item \label{core:vi} For any $X\subseteq V$, we have $\IS_\cC(X) = \{ v\in X\mid r_\cC(X-v)=r_\cC(X)\}$. 
\end{enumerate}
\end{lemma}
\begin{proof}
We prove the lemma by showing \ref{core:i}$\implies$\ref{core:ii}$\implies$\ref{core:iii}$\implies$\ref{core:iv}$\implies$\ref{core:v}$\implies$\ref{core:i} and \ref{core:ii}+\ref{core:v}$\implies$\ref{core:vi}$\implies$\ref{core:i}. 
\vspace{-0.5em}

\paragraph{\ref{core:i}$\implies$\ref{core:ii}} This implication is well-known in matroid theory, see e.g., \cite{Welsh1976}.
\vspace{-0.5em}

\paragraph{\ref{core:ii}$\implies$\ref{core:iii}} By the definition of the closure, we have $\FC_h(X)=\FC_{\Phi_{\cC}}(X)$, and thus we can label the vertices $\FC_h(X)\setminus X=\{v_1,\dots ,v_k\}$ in such a way that the family $\cC$ contains $k$ circuits $C_1,\dots,C_k$ with $C_j\setminus(X\cup \{v_\ell\mid \ell\in[j-1]\}=\{v_j\}$ for $j\in[k]$. This together with \ref{core:ii} implies $r_{\cC}(X)=r_{\cC}(X\cup\{v_1\})=\dots =r_{\cC}(X\cup\{v_j\mid j\in[k]\})=r_{\cC}(\FC_h(X))$, proving \ref{core:iii}.  
\vspace{-0.5em}

\paragraph{\ref{core:iii}$\implies$\ref{core:iv}} By the monotonicity of the rank function, we have $r_\cC(X)\leq r_\cC(X+v) \leq r_\cC(\FC_h(X))$ for all $v\in \FC_h(X)$. Thus it follows from \ref{core:iii} that $r_\cC(X)=r_\cC(X+v)$ for all $v\in \FC_h(X)$, implying that $\FC_h(X) \subseteq \{ v\in V\mid r_\cC(X+v)=r_\cC(X) \}$. 

For the converse inclusion, we shall show that any  $v\in V$ with $r_\cC(X+v)=r_\cC(X)$ belongs to $\FC_h(X)$. Let $B$ be a set in $\cC^{dc}$  such that $r_\cC(X)=|B\cap X|$. By definition, $B\cap X$ is an independent set of $\cC$. Note that $(B \cap X) +v$ is dependent, i.e., there exists a $C \in \cC$ with $v\in C  \subseteq (B \cap X) +v$, since otherwise  there exists a maximal independent set $B' \in \cC^{dc}$ that contains $(B \cap X) +v$, which implies $r_\cC(X+v)>r_\cC(X)$. The existence of such a $C$ implies $v \in \FC_h(X)$, which completes the proof. 
\vspace{-0.5em}

\paragraph{\ref{core:iv}$\implies$\ref{core:v}} Let $B$ be a set in $\cC^{dc}$ such that $r_\cC(X)=|B\cap X|$. Note again that $B\cap X$ is an independent set of $\cC$. By \ref{core:iv}, any $v \in \FC_h(X)$ satisfies that $(B \cap X) +v$ is dependent, which implies the existence of  $C \in \cC$ with $v\in C  \subseteq (B \cap X) +v$. This implies $ v \in \FC^1_{\Phi_\cC}(X)$, completing the proof.  
\vspace{-0.5em}

\paragraph{\ref{core:v}$\implies$\ref{core:i}} Axioms \ref{cax1} and \ref{cax2} hold by the definition of Sperner hypergraphs. To prove the third axiom \ref{cax3}, let us consider two hyperedges $C_1,C_2\in\cC$ with $v\in C_1\cap C_2$. Sine $\cC$ is Sperner, there exists a  variable $u$ in $C_2\setminus C_1$. Define $X=C_1\cup C_2\setminus\{u,v\}$. Then we have $u\in\FC_h(X)$, thus $u\in \FC^1_{\Phi_\cC}(X)$ by \ref{core:v}, implying the existence of a hyperedge $C_3\in\cC$ such that $u\in C_3\subseteq X+u$. Since $X+u=C_1\cup C_2-v$, our claim follows.
\vspace{-0.5em}

\paragraph{\ref{core:ii}+\ref{core:v}$\implies$\ref{core:vi}} Let $X$ be a  subset of $V$. For any $v \in \IS_\cC(X)$, we have  $\IS_\cC(X)-v \overset{h}{\rightarrow} v$. By \ref{core:v}, then there exists a hyperedge $C\in\cC$ such that $v\in C\subseteq \IS_\cC(X)$. Since $C\setminus (X-v)=\{v\}$, it follows from \ref{core:ii} that $r_\cC(X-v)=r_\cC((X-v)+v)=r_\cC(X)$, which means that $v$ belongs to the set in the right-hand side of \ref{core:vi}. 

On the other hand, if  $v\in X$ satisfies  $r_\cC(X-v)=r_\cC(X)$, then there exists a  $B \in \cC^{dc}$ such that $|B \cap X|=|B \cap (X-v) |=r_\cC(X)$. By this, there exists a hyperedge $C\in\cC$ such that $\{v\} \in C\subseteq B\cap X$, which implies that $v\in \IS_\cC(X)$, since $\IS_\cC(X)$ is the unique maximal implicate set contained in $X$. 
\vspace{-0.5em}

\paragraph{\ref{core:vi}$\implies$\ref{core:i}} To see that $\cC$ satisfies the circuit axioms \ref{cax3}, let $C_1$ and $C_2$ be two hyperedges in $\cC$ with $v\in C_1 \cap C_2$ for some $v \in V$. We shall show that  $C_1\cup C_2-v$ contains  a hyperedge in $\cC$, which completes the proof. If it is not the case, then there exists a maximal independent set $B$ in $\cC^{dc}$ that contains $C_1\cup C_2-v$. This implies that  $ r_\cC(C_1\cup C_2-v)=|C_1\cup C_2|-1$. Since $\IS_\cC(C_1\cup C_2)=C_1\cup C_2$,  this together with \ref{core:vi} implies
\begin{equation*}
%\label{eq-aaa1}
r_\cC(C_1\cup C_2)=r_\cC(C_1\cup C_2-u)=|C_1\cup C_2|-1 \mbox{ for any }u \in C_1 \cup C_2, 
\end{equation*}
which  further implies that for any $u \in C_1 \cup C_2$ there exists a set $B_u \in \cC^{dc}$ with $B_u \supseteq  C_1\cup C_2-u$. Note however that $C_1\setminus C_2\neq\emptyset$, since $\cC$ is Sperner. This means that  $C_2\subseteq C_1\cup C_2-u\subseteq B_u$ for any $u \in C_1 \setminus C_2$, a contradiction. 
\end{proof}

We show next  an important claim that will be instrumental in the proof of one of our main theorems. We here remark that it can be regarded as implicate-duality of hypergraphs $\cC$ and $\cC^{dcdc}$, which corresponds to  Theorem~\ref{thm:summary}\ref{e2-true-set-of-h}. 

\begin{lemma}\label{l-nonmatroidal-C}
A Sperner hypergraph  $\cC\subseteq 2^V$ satisfies the circuit axioms \ref{cax3} if and only if  $|C\setminus T|\neq 1$ for all $C \in \cC$ and all $T\in \cC^{dcdc}$.
\end{lemma}
\begin{proof}
If $\cC$ satisfies the circuit axiom \ref{cax3}, then Table \ref{tab:table1} and Lemma \ref{lemma-matroid-formul-part1}\ref{p1:b} implies that $\cC^{dcdc}$ consists of the family of hyperplanes of the matroid $\bM$ defined by $\cC$. It is known that $|C \setminus T| \not=1$ holds for any circuit $C\in \cC$ and any hyperplane $T \in \cC^{dcdc}$, which completes the proof of the only-if statement. 

Suppose that the hypergraph $\cC$ violates the circuit axiom \ref{cax3}. Then there exist two hyperedges $C_1,C_2\in\cC$ and a variable $v\in C_1\cap C_2$ such that the set $S=C_1\cup C_2-v$  contains no hyperedge in $\cC$.  This implies that $S$ is an independent set of  $\cC$, and thus it is contained in a maximal independent set $B\in \cC^{dc}$. Since $B \not\supseteq C_1$, we have $v\notin B$. For a variable $u\in C_2\setminus C_1$, there exists a maximal independent set $T$ of $\cC^{dc}$ (i.e., $T \in \cC^{dcdc}$) with $u \not\in T$. Note that such $u$ and $T$ must exist, since $\cC$ and $\cC^{dc}$ are both Sperner. Now, this set $T$ may or may not contain vertex $v$. If $v\in T$, then we have $C_2\setminus T=\{u\}$, while if $v\not\in T$, then we have $C_1\setminus T=\{v\}$. This completes the proof of the if statement. 
\end{proof}

Let us remark that Theorem~\ref{t-matroidal-formulaic-equivalences}\ref{fo-equiv:iv} and Lemma \ref{l-nonmatroidal-C} implies that, for a hypergraph Horn function $h=\Phi_\cC$, the family $\cC^{dcdc}$ contains only true sets of $h$ if and only if $h$ is matroid Horn and $\Phi_\cC$ is its canonical representation. Furthermore, in this case the family $\cC^{dcdc}$ is the set of characteristic models of $h$ by Lemma \ref{lemma-matroid-formul-part1}\ref{p1:c}. 

We are now ready to prove our first main theorem.

\begin{proof}[\textbf{Proof of Theorem \ref{t-matroidal-formulaic-equivalences}}]
	
We prove the theorem by showing \ref{fo-equiv:i}$\implies$\ref{fo-equiv:ii}$\implies$\ref{fo-equiv:iii}$\implies$\ref{fo-equiv:i} and \ref{fo-equiv:i}$\implies$\ref{fo-equiv:iv}$\implies$\ref{fo-equiv:i}. 
\vspace{-0.5em}

\paragraph{\ref{fo-equiv:i}$\implies$\ref{fo-equiv:ii}} The implication follows from Lemma~\ref{lemma-matroid-formul-part1}\ref{p1:a}. 
\vspace{-0.5em}

\paragraph{\ref{fo-equiv:ii}$\implies$\ref{fo-equiv:iii}} The implication follows from  Lemma \ref{l-E-M(h)-K(h)} which claims that $\cM(h)=\cK(h)^{dc}$ for all definite Horn functions, since $\cK(h)=\cC^{dc}$.
\vspace{-0.5em}

\paragraph{\ref{fo-equiv:iii}$\implies$\ref{fo-equiv:i}} Suppose that $\cC$ violates the circuit axiom \ref{cax3}. Then, by Lemma \ref{l-nonmatroidal-C}, we have $C\in \cC$ and $T\in\cC^{dcdc}$ such that $|C\setminus T|=1$. However, this means that  $T$ is a false set of $h$ by Theorem~\ref{thm:summary}\ref{e2-true-set-of-h}, which completes the proof.
\vspace{-0.5em}

\paragraph{\ref{fo-equiv:i}$\implies$\ref{fo-equiv:iv}} The implication follows from Lemma \ref{lemma-matroid-formul-part1}\ref{p1:c}. 
\vspace{-0.5em}

\paragraph{\ref{fo-equiv:iv}$\implies$\ref{fo-equiv:i}} If $\cC$ violates the circuit axiom \ref{cax3}, then Lemma \ref{l-nonmatroidal-C} implies the existence of $C\in \cC$ and $T\in \cC^{dcdc}$ such that $|C\setminus T|=1$. This means that $T$ is a false set of $h$ by Theorem~\ref{thm:summary}\ref{e2-true-set-of-h}, completing the proof.
\end{proof}

Let $\bM$ be a matroid defined by its circuits $\cC$. The so-called dual matroid of $\bM$ has circuits $\cC^{dcd}$, see e.g.\ \cite{Welsh1976}. This together with implicate-duality of hypergraph Horn functions imply the following characterizations.  

\begin{lemma}\label{t-matroid-dual}
Let $\cC\subseteq 2^V$ be a nontrivial Sperner hypergraph, and let $h$ be the hypergraph Horn function represented by $\Phi_\cC$. Then the following are equivalent.
\begin{enumerate}[label={\rm (\roman*)}]\itemsep0em
%	\item[(i)] $\cM(h)^\cap=\cT(h)$
    \item \label{dual:i} $\cC$ satisfies the circuit axiom \ref{cax3}.
    \item \label{dual:ii} $\cC^{dcd}$ satisfies the circuit axiom \ref{cax3}.
	\item \label{dual:iii} $h^i=\Phi_{\cC^{dcd}}$.
	\item \label{dual:iv} $\cK(h^i)=\cC^{d}$.
	\item \label{dual:v} $\cM(h^i)=\cC^c$.
	\item \label{dual:vi} $\cT(h^i) = (\cC^{c})^{\cap}$. 
\end{enumerate}
\end{lemma}
\begin{proof}
It is known that $\cC$ satisfies the circuit axiom \ref{cax3} if and only if $\cC^{dcd}$ satisfies it ~\cite{Welsh1976}, which shows the equivalence of \ref{dual:i} and \ref{dual:ii}. The other equivalences can be obtained by applying Theorem \ref{t-matroidal-formulaic-equivalences} to the hypergraph $\cC^{dcd}$. 
\end{proof}

We are now ready to prove our second main result in this section. 

\begin{proof}[\textbf{Proof of Theorem \ref{t-matroidal-functional-equivalences}}]

We prove the theorem by showing \ref{fu-equiv:i}$\implies$\ref{fu-equiv:ii}$\implies$\ref{fu-equiv:i}, \ref{fu-equiv:i}$\iff$\ref{fu-equiv:iii} and \ref{fo-equiv:i}$\iff$\ref{fu-equiv:iv}. 
\vspace{-0.5em}

\paragraph{\ref{fu-equiv:i}$\implies$\ref{fu-equiv:ii}} Let $\cC$ be a hypergraph satisfying the circuit axiom \ref{cax3}, and let $h$ be a matroid Horn function represened by $\Phi_{\cC}$. For a prime implicate $S\to v$ of $h$,  we have $v\in \FC_h(S)$, and thus Lemma \ref{t-new-core}\ref{core:v} implies that $v\in \FC_{\Phi_\cC}^1(S)$. This is equivalent to the existence of a hyperedge $C\in\cC$ such that $C\setminus S=\{v\}$. By the primeness of $S\to v$, we have  $S=C-v$. Consequently, all prime implicates of $h$ are included in $\Phi_\cC$. Furthermore, since $\cC$ is Sperner, all clauses in $\Phi_{\cC}$ are prime implicates of $h$ by Lemma \ref{t-new-core}\ref{core:v}, which shows that the complete CNF of $h$ is indeed the circular CNF.
\vspace{-0.5em}

\paragraph{\ref{fu-equiv:ii}$\implies$\ref{fu-equiv:i}} Assume that the complete CNF of $h$ is the circular CNF $\Phi_{\cC}$ for a  hypergraph $\cC\subseteq 2^V$. Note that $\cC$ is Sperner, since otherwise it contains a non-prime implicate of $h$. We then claim that $\cC$ satisfies the circuit axiom \ref{cax3}. 

Let us consider two hyperedges $C_1$ and $C_2$ in $\cC$ with  $v\in C_1\cap C_2$ for some $v \in V$. Since $\cC$ is Sperner, there exists a variable $u \in C_2\setminus C_1$. Then for the set $X=(C_1\cup C_2)\setminus\{u,v\}$ we have $u\in \FC_{\Phi_{\{C_1,C_2\}}}(X)\subseteq \FC_{\Phi_{\cC}}(X)$. By Lemma \ref{l-E-implicate},  the clause $X\to u$ is an implicate of $h=\Phi_{\cC}$. Therefore there exists a prime implicate $S\to u$ of $h$ for which $S\subseteq X$.  Since $(S\to u)$ is a clause in $\Phi_{\cC}$  by our assumption, we have $C_3=S+u\in\cC$. This completes the proof, since $C_3\subseteq (C_1\cup C_2)\setminus\{v\}$.  
\vspace{-0.5em}

\paragraph{\ref{fu-equiv:i}$\iff$\ref{fu-equiv:iii}} The implications follow from Lemma \ref{t-matroid-dual} and the duality of matroids.
\vspace{-0.5em}

\paragraph{\ref{fo-equiv:i}$\iff$\ref{fu-equiv:iv}} The implications follow by applying the equivalence of \ref{fu-equiv:i} and \ref{fu-equiv:ii} shown above to $h^i$.
\end{proof}

The above results show that the notion of implicate duality of hypergraph Horn functions generalizes the notion of matroid duality. 

Let us add further remarks on the relation of $\cM(h)$ and the characteristic models of $h$. We have seen that for a matroid Horn function $h$, these two sets coincide. However, the same may happen even for a hypergraph Horn function $h$ that is not matroid Horn. Note that if we do not insist on $h$ being hypergraph Horn, then finding such an example is simple. Namely, let $h'$ be an arbitrary definite Horn function, and consider $\cM=\cM(h')$. Now define $\cT$ as the intersection closure $\cM(h')$. By \cite{McKinsey1943}, there exists a unique definite Horn function $h$ such that $\cT$ is its set of true sets. Then, by our construction, $\cM=\cM(h)$ is the set of characteristic models of $h$. The next example shows that analogous examples can exist among hypergraph Horn functions too.

\begin{ex}\label{e-nonmatroidal-M-char}
Let $V=\{0,1,2,3,4,5\}$, $E=\{\{i,(i+1\bmod 5)\}\mid i=0,1,\dots,4\}\cup \{\{i,5\}\mid i=0,1,\dots,4\}$, and $G=(V,E)$. We view the edges of $G$ as subsets of $V$ of size two. Note that $G$ has five triangles, and we use those to define the hypergraph $\cH=\{\{i,(i+1 \bmod 5),5\}\mid i=0,1,\dots,4\}$. Finally, let $h=\Phi_{\cH}$ be the hypergraph Horn function associated to $\cH$. 

It is not difficult to check that the minimal keys of $h$ are exactly the edges of $G$. First we show that $h$ is a not matroidal. To see this, it suffices to show that $\cK(h)$ violates the basis exchange axioms. For instance, edges $K_1=\{0,1\}$ and $K_2=\{2,3\}$ are members of $\cK(h)$, but vertex $v=1\in K_1$ cannot be exchanged to any of the vertices of $K_2$, since neither $\{0,2\}$ nor $\{0,3\}$ are edges of $G$. 

Now we show that $\cM(h)$ is the set of characteristic models of $h$. As $\cK(h)=E$, we have $\FC_h(S)=V$ for any subset $S\subseteq V$ that contains an edge. This implies that nontrivial closed sets cannot contain edges of $G$, hence the nontrivial true sets of $h$ are the nonempty independent sets of $G$. An easy computation shows that $\cM(h)=\cK(h)^{dc}$ is the set of maximal independent sets of $G$, i.e., we have $\cM(h)=\{ 5 \} \cup \{\{i,(i+2 \bmod 5)\}\mid i=0,1,\dots,4\}$. From this, one can derive $\cM^{\cap}(h)=\cM(h)\cup \{\{i\}\mid i=0,1,\dots,4\}\cup \{\emptyset\}$. These together imply that $\cM(h)$ is the set of characteristic models of $h$, since $\cM^{\cap}(h)$ is exactly the family of independents sets of $G$. 
\end{ex}

Our next example shows that matroid Horn and non-matroidal hypergraph Horn functions may have the same set of minimal keys, while by Theorem~\ref{t-matroidal-formulaic-equivalences} a hypergraph can be the set of minimal keys of at most one matroid Horn function. 

\begin{ex}
Let $V=X\cup Y$ where $X=\{1,2,3,4\}$ and $Y=\{5,6,7,8\}$. Furthermore, define $E=\{\{1,3\},\{1,4\},\{2,3\},\{2,4\}\}$ and $F=\{\{5,7\},\{5,8\},\{6,7\},\allowbreak\{6,8\}\}$. Consider the Sperner family $\cH=\{\{1,2,3,4\},\{3,4,5,6\},\{5,6,7,8\}\}$. It is not difficult to check that 
\begin{equation*}
\cK(\Phi_{\cH}) = \left\{ K\subseteq V  \bigm| |K\cap X|=3 ~\text{and}~ K\cap Y\in F\right\}\cup
\left\{ K\subseteq V  \bigm| |K\cap Y|=3 ~\text{and}~ K\cap X\in E\right\}.
\end{equation*}
From this, one can derive that 
\[
\cC  =  \cK(\Phi_{\cH})^{cd}  =  \{X,Y\} \cup \left\{\{a,b\},\{c,d\} \right\} \times \left\{\{e,f\},\{g,h\}\right\}
\]
and that it satisfies the circuit axioms \ref{cax1}--\ref{cax3}. Thus, by Table~\ref{tab:table1}, we have $\cK(\Phi_{\cC}) = \cC^{dc} =  \left( \cK(\Phi_{\cH})^{cd}\right)^{dc} = \cK(\Phi_{\cH})$, since both operators ``$c$" and ``$d$" are involutions over the set of Sperner hypergraphs. However, $\{1,2,3,5\}{\rightarrow} 6$ is a prime implicate of $\Phi_\cH$ while $\{1,2,5,6\}\overset{\Phi_{\cH}}{\nrightarrow}3$, and hence $\Phi_\cH$ is not matroid Horn by Theorem~\ref{t-matroidal-functional-equivalences}.
\end{ex}

%%%%%%%%%%%%%%%%
\section{Minimum representations}
\label{sec:representation}
%%%%%%%%%%%%%%%%

Given a system $\cF\subseteq 2^V$ and a subsystem $\cG\subseteq \cF$, let $\langle\cG\rangle^1_\cF$ denote the family of sets in $\cF$ that $\cG$ \emph{generates in a single step}, that is,
\begin{equation*}
\langle\cG\rangle^1_\cF=\cG\cup\{X\in\cF\mid X=(X_1\cup X_2)-v\ \text{for distinct}\ X_1,X_2\in\cG,\ v\in X_1\cap X_2\}.
\end{equation*}
We denote the repeated application of operator $\langle\ \cdot\ \rangle^1_\cF$ for $k\geq 2$ times by $\langle\ \cdot\ \rangle^k_\cF$, that is,
\begin{equation*}
\langle\cG\rangle^k_\cF=\langle\langle\cG\rangle^{k-1}_\cF\rangle^1_\cF.
\end{equation*}
Since $V$ is finite, for some $k$ we have $\langle\cG\rangle^k_\cF=\langle\cG\rangle^{k+1}_\cF$. We denote this final system by $\langle\cG\rangle_\cF$ and call $\cG$ a \emph{generator} of $\cF$ if $\langle\cG\rangle_\cF=\cF$.

Let $\bM=(V,\cC)$ be a matroid and $h_\bM$ be the corresponding matroid Horn function. For a circuit $C\in\cC$ and $v\in C$, the clause $(C-v)\rightarrow v$ is called a \emph{circuit clause}. Our goal is to find compact representations of $h_\bM$, and therefore of $\bM$ as well. To this end, we consider three different objectives:
\begin{description}\itemsep0em
\item[\textbf{(G) circuit generator:}] $|\bM|_G=$ minimum cardinality of a generator of $\cC$, 
\item[\textbf{(C) number of circuits:}]  $|\bM|_C=$ minimum cardinality of a subsystem $\cD\subseteq \cC$ s.t. $h_\bM=\Phi_{\cD}$,
\item[\textbf{(K) number of circuit clauses:}] $|\bM|_{K}=$ minimum number of circuit clauses needed to represent $h_\bM$.
\end{description}
Objective (G) characterizes the complexity of the circuit family of $\bM$. The number of circuits (C) denotes the minimum size of a subsystem of $\cC$ for which the corresponding hypergraph Horn CNF provides a representation of $h_\bM$. The number of clauses is an important parameter for SAT solvers when a Horn formula encodes a constraint which is part of a larger problem; objective (K) captures an analogous notion when the set of usable clauses is restricted to circuit clauses.

\subsection{Binary matroids}
\label{sec:binary}

A matroid $\bM=(V,\cC)$ is called \emph{binary} if it can be represented over the finite field GF(2), or in other words, its elements corresponds to the columns of a $0,1$-matrix such that a set $X\subseteq V$ is independent in $\bM$ if and only if the corresponding columns are linearly independent over GF(2). By abuse of notation, we denote the binary matrix representing the matroid also by $\bM$, and for a subset $X\subseteq V$, the sum of the corresponding columns of $\bM$ is written as $\sum_{u\in X} u$. The class of binary matroids contains several important subclasses, such as regular or graphic matroids.

We will need the following easy observations on binary matroids.

\begin{lemma}\label{lem:binary}
Let $\bM=(V,\cC)$ be a binary matroid and $X\subseteq V$ be an independent set. Then there is at most one $v\in V$ for which $X+v$ forms a circuit of $\bM$. 
\end{lemma}
\begin{proof}
Let $v=\sum_{u\in X}u$. We claim that if $v$ is contained in $V$, then it is the unique element for which $X+v$ forms a circuit. The set $X+v$ is clearly dependent, but any proper subset $X'\subsetneq X+v$ is independent. Indeed, if $X'\subsetneq X+v$ is dependent, then necessarily $v\in X'$. But then $(X+v)\setminus X'\subseteq X$ is a non-empty dependent subset of $X$, a contradiction.
\end{proof}

\begin{lemma}\label{lem:binary2}
Let $\bM=(V,\cC)$ be a binary matroid without parallel elements. If $C_1,C_2\in\cC$ are such that $|C_1\setminus C_2|=1$, then $|C_1|<|C_2|$.
\end{lemma}
\begin{proof}
As $C_2\nsubseteq C_1$ and $|C_1\setminus C_2|=1$, we have $|C_1|\leq |C_2|$. Suppose indirectly that $|C_1|=|C_2|$, implying that $|C_2\setminus C_1|=1$. Let $C_1\setminus C_2=\{v\}$ and $C_2\setminus C_1=\{w\}$. The intersection $X=C_1\cap C_2$ is an independent set of $\bM$ for which both $X+v$ and $X+w$ are circuits. By Lemma~\ref{lem:binary}, $v=w=\sum_{u\in X} u$, contradicting the assumption that $\bM$ does not contain parallel elements. 
\end{proof}

Motivated by the example of graphic matroids, we call a circuit $C\in\cC$ of a simple binary matroid \emph{chordless} if there exists no $C'\in\cC$ such that $|C'\setminus C|=1$ and $|C'|<|C|$. Our first result characterizes the minimum size of a generator of $\cC$, i.e. objective (G), for simple binary matroids.

\begin{theorem} \label{thm:binary_g}
Let $\bM=(V,\cC)$ be a simple binary matroid. Then the set of chordless cycles is the unique minimum generator of $\cC$.
\end{theorem}
\begin{proof}
Let $C_1,C_2\in C$ and $v\in C_1\cap C_2$. If $C=(C_1\cup C_2)-v$ is a circuit of $\bM$, then $|C_1\setminus C|=|\{v\}|=1$ and so $|C_1|<|C|$ by Lemma~\ref{lem:binary2}. That is, every circuit obtained by the generation step is non-chordless, hence all the chordless circuits must be contained in any generator of $\bM$.

Let us denote by $\cD$ the set of chordless cycles of $\bM$, and take an arbitrary circuit $C\in\cC$. We prove by induction on the size of $|C|$ that it can be generated from $\cD $, that is, $C\in\langle\cD \rangle_\cC$. This clearly holds for chordless circuits. Assume now that $C\in\cC$ is not chordless. By definition, there exists $C_1\in\cC$ such that $|C_1\setminus C|=1$ and $|C_1|<|C|$. Let $C_1\setminus C=\{v\}$ and take an arbitrary element $w\in C\cap C_1$. Notice that such an element exists as $|C_1|>1$ by the simplicity of $\bM$. By the third circuit axiom \ref{cax3}, there exists a circuit $C_2\subseteq (C\cup C_1)-w$. As $C_2\nsubseteq C$, we have $v\in C_2$ and so $|C_2\setminus C|=1$. Therefore $|C_2|<|C|$ by Lemma~\ref{lem:binary2}. By the induction hypothesis, both $C_1$ and $C_2$ can be generated from $\cD $, implying that $C$ can be generated as well. 
\end{proof}

It turns out that the set of chordless circuits provides an optimal representation in terms of $|\cdot|_C$ and $|\cdot|_K$ as well. The proof is analogous to that of Theorem~\ref{thm:binary_g}.

\begin{theorem} \label{thm:binary_ck}
Let $\bM=(V,\cC)$ be a simple binary matroid. Then the set of chordless cycles is the unique minimum subsystem of $\cC$ with respect to $|\bM|_C$. Furthermore, $\{(C-v)\rightarrow v\mid C\in\cC\ \text{is chordless},\ v\in C\}$ is the unique minimum set of circuit clauses with respect to $|\bM|_K$.
\end{theorem}
\begin{proof}
Let us denote by $\cD$ the set of chordless cycles of $\bM$. Take a chordless circuit $C\in\cD $, and consider an arbitrary $u\in C$. Then $(C-u)\rightarrow u$ is an implicate of $h_\bM$, hence for any representation of $h_\bM$ there exists a circuit $C'$ and $v\in C'$ such that $C'-v\subseteq C-u$. If $v\in C$ then necessarily $C=C'$ by circuit axiom \ref{cax2}. Otherwise $|C'\setminus C|=1$, and so $|C'|<|C|$ by Lemma~\ref{lem:binary2}, contradicting $C$ being chordless. This implies that every representation of $h_\bM$ contains the circuit clause $(C-u)\rightarrow u$.

We claim that $h_\bM=\Phi_{\cD }$. To see this, take an arbitrary circuit $C\in\cC$ and $u\in C$. We prove by induction on the size of $|C|$ that $(C-u)\rightarrow u$ is an implicate of $\Phi_{\cD }$. This clearly holds for chordless circuits. Assume now that $C\in\cC$ is not chordless. By definition, there exists $C_1\in\cC$ such that $|C_1\setminus C|=1$ and $|C_1|<|C|$. Let $C_1\setminus C=\{v\}$ and take an arbitrary element $w\in C\cap C_1$. Notice that such an element exists as $|C_1|>1$ by the simplicity of $\bM$. By the third circuit axiom \ref{cax3}, there exists a circuit $C_2\subseteq (C\cup C_1)-w$. As $C_2\nsubseteq C$, we have $v\in C_2$ and so $|C_2\setminus C|=1$. Therefore $|C_2|<|C|$ by Lemma~\ref{lem:binary2}. By the induction hypothesis, $(C_1-v)\rightarrow v$ and $(C_2-u)\rightarrow u$ are both implicates of $\Phi_{\cD}$,  implying that $(C-u)\rightarrow u$ is also an implicate. 
\end{proof}

\begin{remark}
The gap between the number of circuits and the number of chordless circuits can be exponentially large. For example, consider a complete bipartite graph $G=(S,T;E)$ with $|S|=|T|=n$, and let $\bM=(E,\cC)$ be the graphic matroid of $G$. Then the number of cycles in $G$ is clearly exponential in $n$, while the number of chordless cycles is ${n \choose 2}^2$.
\end{remark}

\subsection{Uniform matroids}
\label{sec:uniform}

Given a positive integer $r$, the \emph{uniform matroid} $\bM=(V,\cC)$ of rank $r$ is defined by $\cC=\{C\subseteq V\mid |C|=r+1\}$. Although uniform matroids form one of the simplest matroid classes, determining the exact value of $|\bM|_C$ seems to be a difficult problem. For $|\bM|_G$ and $|\bM|_K$ we give closed formulas for the case of uniform matroids. It is worth mentioning that $h_\bM$ in this case is a key Horn function~\cite{berczi2022approximating} for which $F_{h_\bM}(X)=X$ if $|X|\leq r-2$ and $F_{h_\bM}(X)=V$ otherwise.

\begin{theorem} \label{thm:uniform_g}
Let $\bM=(V,\cC)$ be a rank-$r$ uniform matroid. Then $|\bM|_G=n-r$.
\end{theorem}
\begin{proof}
Let $v_1,\dots,v_n$ denote the elements of the ground set $V$. Consider the system $\cD =\{\{v_i,\dots,v_{i+r}\}\mid i\in[n-r]\}$. We claim that $\cD $ is a generator of $\cC$. Take an arbitrary $C\in\cC$, and assume that $C=\bigcup_{j=1}^q\{v_{i_j},\dots,v_{i_j+\delta_j}\}$, where $i_j+\delta_j<i_{j+1}$ for $j\in[q-1]$. We prove by induction on $\delta_1$, and with respect to this, on $q$ that $C$ can be generated from $\cD $. If $\delta_1=0$ then $q=1$ and $C\in\cD $, hence assume that $\delta_1>0$. Let $C_1=C-v_{i_q+\delta_q}+v_{i_1+\delta_1+1}$ and $C_2=C-v_1+v_{i_1+\delta_1+1}$. It is not difficult to see that either $\delta_1$ or $q$ decreases in both cases, hence the circuits $C_1$ and $C_2$ can be generated starting from $\cD $. Notice that $C=(C_1\cup C_2)-v_{i_1+\delta_1+1}$, therefore $C$ can be generated as well.

To see that $\cD $ is an optimal solution, consider any generator $\cG$. Let us call a pair $C_1,C_2\in\cG$ connected if $|C_1\setminus C_2|=|C_2\setminus C_1|=1$. Accordingly, one can define the connected components $\cG_1,\dots,\cG_s$ of $\cG$. We denote the number of circuits in $\cG_i$ by $q_i$. By the definition of connectivity, the union of the circuits in $\cG_i$ contains at most $r+q_i$ elements of $V$. Moreover, $|\bigcup_{C\in\cG_i}C\cap(\bigcup_{C\in\cG\setminus\cG_i} C)|\geq r$ for $i\in[p]$. Indeed, if this does not hold then the generation step cannot be applied to $C_1\in\langle\cG_i\rangle_\cC$ and $C_2\in\langle\cG\setminus\cG_i\rangle_\cC$ as $|C_1\cap C_2|\leq r-1$. This implies that no circuit $C$ with $|C\cap\bigcup_{C\in\cG\setminus\cG_i}|=r$ can be generated, contradicting $\cG$ being a generator. As the circuits in $\cG$ must cover all the $n$ elements of $V$, we get
\begin{align*}
n
{}&{}\leq 
\left|\bigcup_{C\in\cG_1}C\right|+\sum_{i=2}^s\left|\left(\bigcup_{C\in\cG_i}C\right)\setminus \left(\bigcup_{C\in\cG\setminus\cG_i} C\right)\right|\\
{}&{}\leq
(r+q_1)+\sum_{i=2}^s [r+q_i-r]\\
{}&{}=
r+\sum_{i=1}^s q_i.
\end{align*}
This implies $|\cG|=\sum_{i=1}^s|\cG_i|=\sum_{i=1}^s q_i\geq n-r$, concluding the proof of the theorem.
\end{proof}

Let us now turn to the case of $|\cdot|_C$. We give two results, the first being based on a counting argument, while the second showing a connection to Tur\'an systems.

\begin{theorem} \label{thm:uniform_c_1}
Let $\bM=(V,\cC)$ be a rank-$r$ uniform matroid on $n\geq r+2$ elements. Then ${n\choose r}/(r+\frac{1}{2})\leq |\bM|_C\leq {n\choose r}$.
\end{theorem}
\begin{proof}
We prove the lower bound by a token counting argument. Let us give one token to every $r$-element subset of $V$, and consider an arbitrary subsystem $\cD $ of $\cC$ such that $\Phi_{\cD }=h_\bM$. We redistribute the tokens as follows: every $r$-element set shares its token evenly among the sets in $\cD $ containing it. Then every set in $\cD $ receives at most $r+\frac{1}{2}$ tokens in total. Indeed, for every circuit $C_1\in\cD $ there must be another circuit $C_2\in\cD $ with $|C_1\cap C_2|=r$, as otherwise $F_{\Phi_{\cD}}(C)=C\neq V$ by $r<n-1$, a contradiction. This proves the lower bound.

The upper bound follows from an easy construction: pick an arbitrary element $v\in V$ and consider the system $\cD =\{X+v\mid X\subseteq V-v,|X|=r\}$. It is not difficult to check that $F_{\Phi_{\cD }}(X)=X$ if $|X|\leq r-1$ and $F_{\Phi_{\cD }}(X)=V$ otherwise, hence $\Phi_{\cD }=h_\bM$ as required.
\end{proof}

Interestingly, for uniform matroids, the objective $|\cdot|_C$ is closely related to Tur\'an and covering numbers. A \emph{Tur\'an $(n,t,k)$-system} is an $k$-uniform hypergraph of $n$ vertices such that every $t$-element subset of vertices contains at least one hyperedge. The \emph{Tur\'an number} $t(n,t,k)$ asks for the minimum size of such a hypergraph. Determining the exact value of $t(n,t,k)$ is a problem posed by Tur\'an \cite{turan-61}. The simplest case $t=3$, $k=2$ was answered by Mantel's theorem \cite{mantel1906wiskundige} from 1907, showing that the largest triangle-free graph on $n$ vertices is a complete bipartite graph with color classes having sizes $\lfloor \frac n2 \rfloor$ and $\lceil \frac n2 \rceil$. The complementary family of a Tur\'an $(n,t,k)$-system is called a \emph{covering $(n,n-k,n-t)$-system}. In other words, a covering $(n,q,r)$-system is a $q$-uniform hypergraph such that every $r$ element subset is contained in at least one of the hyperedges. The minimum size of such a system is denoted by $c(n,q,r)$, which by definition is the same as $t(n,n-r,n-q)$.

Let $\ell(n,q,r)= \left\lceil\frac{n}{q}\left\lceil\frac{n-1}{q-1}\left\lceil\dots\left\lceil\frac{n-r+1}{q-r}\right\rceil\dots\right\rceil\right\rceil\right\rceil$. In \cite{schonheim1964coverings}, Sch\"onheim verified that 
$c(n,q,r)\geq \ell(n,q,r)$ for all $n\geq q\geq r\geq 1$. The special case when $q=3$ and $r=2$ was studied by Fort Jr. and Hedlund~\cite{fort1958minimal}, who showed that Sch\"onheim's bound is tight, implying
\begin{equation}
c(n,3,2)=\begin{cases}
n^2/6 & \text{if $n\equiv 0$}\\
(n^2-n)/6 & \text{if $n\equiv 0$}\\
(n^2+2)/6 & \text{if $n\equiv 2$ or $4$}\\
(n^2-n+4)/6 & \text{if $n\equiv 5$}
\end{cases}
\ (\bmod\ 6). \label{eq:fh}
\end{equation}
For further details on covering and Tur\'an systems, we refer the interested reader to \cite{hartman1986covering} and \cite{sidorenko1995we}, respectively.

We consider two further set families with different structural properties. Given a finite set $V$ with $|V|=n$, let $\cH\subseteq 2^V$ be an $(r+1)$-uniform hypergraph. We call $\cH$ a \emph{Steiner $(n,r+1,r)$-system} if every subset $X\subseteq V$ of size $r$ is contained in exactly one hyperedge, and call it a \emph{implication $(n,r+1,r)$-system} if for every subset of $X\subseteq V$ of size at least $r$ there exists a hyperedge $H$ with $|H\setminus X|=1$. We denote by $s(n,r+1,r)$ and $b(n,r+1,r)$, the minimum cardinality of such systems, respectively. The definitions imply that every implication system is a covering system, which in turn is a Steiner system with respect to the same parameters $(n,r+1,r)$. Therefore, we have
\begin{equation*}
s(n,r+1,r)  \leq  c(n,r+1,r)  \leq  b(n,r+1,r). 
\end{equation*}

The interpretation of these systems using Horn-logic is as follows. An $(r+1)$-uniform hypergraph $\cH\subseteq 2^V$ is a covering $(n,r+1,r)$-system if $\FC_{\Phi_\cH}(X)\neq X$ for all $X\subseteq V, |X|=r$, it is a Steiner $(n,r+1,r)$-system if $|\FC_{\Phi_\cH}(X)|=r+1$ for all $X\subseteq V, |X|=r$, and it is an implication $(n,r+1,r)$-system if $\FC_{\Phi_\cH}(X)=V$ for all $X\subseteq V, |X|=r$. In particular, for a rank-$r$ uniform matroid $\bM=(V,\cC)$, a set of circuits $\cD\subseteq\cC$ satisfy $h_\bM=\Phi_{\cD}$ if and only if $\cD$ forms an implication $(n,r+1,r)$-system. Let us remark that a Steiner $(n,r+1,r)$-system is known to exist only for certain combinations of values of $n$ and $r$, see e.g.~\cite{keevash2018counting}.

First we give lower and upper bounds on $|\bM|_C$ in terms of $c(n,r+1,r)$.

\begin{theorem} \label{thm:uniform_c_2}
Let $\bM=(V,\cC)$ be a rank-$r$ uniform matroid on $n\geq r+1$ elements. Then 
\begin{equation*}
c(n,r+1,r)\leq |\bM|_C\leq 2\cdot c(n,r+1,r).
\end{equation*}
\end{theorem}
\begin{proof}
The lower bound follows from the observations that $h_\bM=\Phi_{\cD}$ for some $\cD\subseteq\cC$ if and only if $\cD$ forms an implication $(n,r+1,r)$-system and that $c(n,r+1,r) \leq b(n,r+1,r)$.

For the upper bound, consider a covering $(n,r+1,r)$-system $\cT$ of minimum size. Let $V=\{v_1,\dots,v_n\}$ be a cyclic ordering of the vertices, and for a subset $T\subseteq V$, let $\ell(T)$ denote the length of a longest interval of $T$ with respect to this cyclic ordering. Furthermore, let us define the following mapping $\varphi\colon 2^V \to 2^V$, see also Figure~\ref{fig:shift}. 
\begin{itemize}
    \item If $T$ forms a single interval, then $\varphi(T)$ is obtained by shifting $T$ by one.
    \item If $T$ consists of several intervals, then $\varphi(T)$ is obtained by extending one of its longest intervals by adding the next element, and deleting an element from another interval of $T$.
\end{itemize}
Note that $\ell(\varphi(T))\geq\ell(T)$ holds for $T\subseteq V$, and strict inequality holds if $T$ consists of more than one interval. Furthermore, we have $|\varphi(T)|=|T|$. By the definition of $\cT$, for every subset $X$ of size $r$, there exists a set $\tau(X)\in\cT$ with $X\subseteq \tau(X)$. Observe that $\tau(X)\cup\varphi(\tau(X))\subseteq\FC_{h_\bM}(X)$. Let $\varphi(\cT)=\{\varphi(T)\mid T\in\cT\}$ and define $\cD =\cT\cup\varphi(\cT)$. We claim that $\cD$ represents $h_\bM$, that is, $\Phi_{\cD}=h_\bM$. To see this, we need to show that the closure $\FC_{\Phi_{\cD}}(X)$ of any set $X$ of size $r$ is~$V$. 

First consider the case when $X$ forms a single interval. Observe that $\varphi(\tau(X))\cup\tau(X)$ contains the interval obtained by shifting $X$ by one, hence the closure of any interval of length $r$ is $V$.

Next we show that the closure of an arbitrary set $X$ of size $r$ contains an interval of length $r$ which, together with the previous case, proves the theorem. This follows from the fact that if $X$ is not an interval then $\ell(\tau(X)\cup\varphi(\tau(X)))>\ell(X)$, hence there exists a set $X'\subseteq \tau(X)\cup\varphi(\tau(X))$ of size $r$ with $\ell(X')>\ell(X)$.
\end{proof}

\begin{figure}[t!]
\centering
\begin{subfigure}[t]{0.46\textwidth}
\centering
\includegraphics[width=\linewidth]{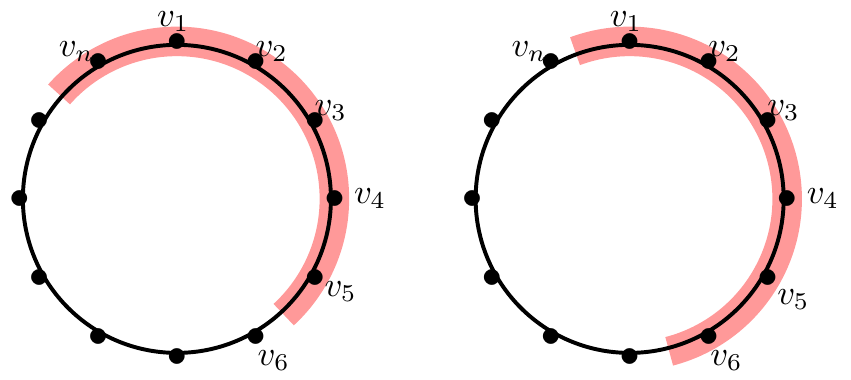}
\caption{The image $\varphi(T)$ is obtained by shifting $T$ by one.}
\label{fig:1}
\end{subfigure}\hfill
\begin{subfigure}[t]{0.52\textwidth}
\centering
\includegraphics[width=0.88\linewidth]{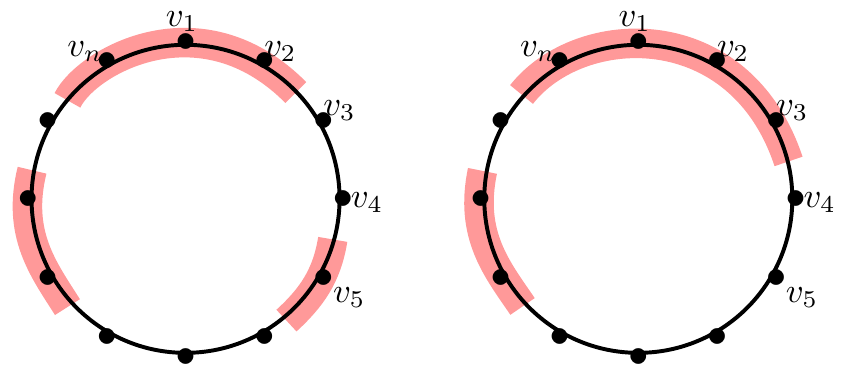}
\caption{The image $\varphi(T)$ is obtained by extending a longest interval and deleting an element from another.}
\label{fig:2}
\end{subfigure}
\caption{Illustration of the mapping $\varphi$ when $T$ forms a single interval and when it consists of several intervals.}
\label{fig:shift}
\end{figure}

Let $\bM$ be a rank-$2$ uniform matroid on $n$ elements. The combination of Theorem~\ref{thm:uniform_c_2} and \eqref{eq:fh} gives $|\bM|_C=\kappa\cdot n^2+O(n)$ for some $1/6\leq\kappa\leq 1/3$, where the lower bound on $\kappa$ can be improved to $1/5$ by Theorem~\ref{thm:uniform_c_1}. Next we prove that the right order of magnitude of $|\bM|_C$ coincides with this lower bound.

\begin{theorem} \label{thm:ranktwo}
Let $\bM=(V,\cC)$ be a rank-$2$ uniform matroid on $n\geq 46$ elements. Then
\begin{equation*}
|\bM|_C=\frac{n^2}{5}+O(n).
\end{equation*}
\end{theorem}
\begin{proof}
Assume that $p$ and $b$ are positive integers such that $n=a\cdot p+b$, $p\geq 3$ and $b\geq 2$. We also assume that $b$ is a constant; its value will be specified later. Given such numbers, we take an arbitrary partition into $p+1$ sets $V=A_1\cup\dots\cup A_p\cup B$ where $|A_i|=5$ for $i\in[p]$ and $|B|=b$. We call the $A_i$s \emph{groups} and the set $B$ as \emph{residual}. Furthermore, we denote the elements of the $i$th group by $A_i=\{a^i_0,\dots,a^i_4\}$ for $i\in[p]$. Throughout the proof, we work with the indices of the elements modulo $5$.

The high-level idea of the construction is as follows. As a first step, we show that there exists a hypergraph $\cC_1$ of size $O(n)$ for which $\FC_{\Phi_{\cC_1}}(\{u,v\})=V$ whenever $\{u,v\}\subseteq A_i\cup B$ for some $i\in[p]$, that is, if $u$ and $v$ are not in distinct groups. Then it remains to consider pairs that have one-one element in two distinct groups. Hence in a second step, we consider a Steiner $(p,3,2)$-system on the set of groups, and define a hypergraph $\cC_2$ of size $n^2/5+O(n)$ which ensures that $|\FC_{\Phi_{\cC_2}}(\{u,v\})\cap A_i|\geq 2$ for some group $i$,  whenever $u$ and $v$ are in distinct groups. For this step, we have to choose the values of $p$ and $b$ carefully in order to ensure the existence of a Steiner $(p,3,2)$-system and keep $b$ to be a constant. At the end, the hypergraph $\cD=\cC_1\cup\cC_2$ has total size $n^2/5+O(n)$ and form an implication $(n,3,2)$-system, proving the claim.

For the first step, define
\begin{align*}
\cC_1
=&
\left\{\{u,v,w\}\mid u,v\in A_i\ \text{for some $i\in[p]$},\ w\in B\right\}\\
\cup&
\left\{\{u,v,w\}\mid u,v\in B,\ w\notin B\right\}.
\end{align*}
By the assumption that $b$ is a constant, we have $|\cC_1|=p\cdot {5 \choose 2}\cdot  b+(n-b)\cdot {b \choose 2}=O(n)$. 

\begin{claim}\label{cl:groups}
Assume that $\{u,v\}\subseteq A_i\cup B$ for some $i\in[p]$. Then $\FC_{\Phi_{\cC_1}}(\{u,v\})=V$.
\end{claim}
\begin{proof}
If $\{u,v\}\subseteq B$, then $w\in\FC_{\Phi_{\cC_1}}(\{u,v\})$ for every $w\in V\setminus B$ since $\{u,v,w\}\in\cC_1$. As $|V\setminus B|=5\cdot p\geq 2$, there exist distinct $w_1,w_2\in V\setminus B$, implying that $z\in\FC_{\Phi_{\cC_1}}(\{u,v\})$ for every $z\in B$ since $\{w_1,w_2,z\}\in\cC_1$.  

If $\{u,v\}\in A_i$ for some $i\in[p]$, then $w\in\FC_{\Phi_{\cC_1}}(\{u,v\})$ for $w\in B$ since  $\{u,v,w\}\in\cC_1$. As $|B|=b\geq 2$, there exist distinct $w_1,w_2\in B$, implying that $z\in\FC_{\Phi_{\cC_1}}(\{u,v\})$ for every $z\in V\setminus B$ since $\{w_1,w_2,z\}\in\cC_1$. 
\end{proof}

For the second step, we need to choose $p$ in such a way that there exists a Steiner $(p,3,2)$-system, usually called a \emph{Steiner triple system}. It is known that such a system exists whenever $p\equiv 1$ or $3\ (\bmod\ 6)$, see more about Steiner triple systems in~\cite{keevash2018counting}. To achieve this, let us fix the value of $b$ such that
\begin{equation*}
b=\begin{cases}
    (n \bmod 30)-15 & \text{if $(n \bmod\ 30)\geq 17$},\\
    (n \bmod 30)+15 & \text{if $(n \bmod\ 30)\leq 16$},
\end{cases}
\end{equation*}
and set $p=(n-b)/5$. By the assumption $n\geq 46$ and the definition of $b$, we get $p\equiv 3\ (\bmod\ 6)$, $p\geq 3$ and $2\leq b\leq 31$. Let $\cT\subseteq 2^{[p]}$ be a Steiner triple system that is defined on the groups. For each triple $T=\{x,y,z\}\in \cT$, set
\begin{align*}
\cC_T
=&
\left\{\{a^x_i,a^y_i,a^z_{i+1}\},\{a^x_i,a^y_i,a^z_{i+2}\}\mid 0\leq i\leq 4\right\}\\
\cup 
&\left\{\{a^y_i,a^z_i,a^x_{i+1}\},\{a^y_i,a^z_i,a^x_{i+2}\}\mid 0\leq i\leq 4\right\}\\
\cup 
&\left\{\{a^x_i,a^z_i,a^y_{i+1}\},\{a^x_i,a^z_i,a^y_{i+2}\}\mid 0\leq i\leq 4\right\},
\end{align*}
and define $\cC_2=\bigcup_{T\in\cT}\cC_T$, see Figure~\ref{fig:cons} for an example. Since $p=n/5+O(1)$, we have $|\cC_2|=30\cdot|\cT|=30\cdot 1/3\cdot {p\choose 2}=\frac{n^2}{5}+O(n)$.

\begin{figure}
    \centering
    \includegraphics[width=0.4\textwidth]{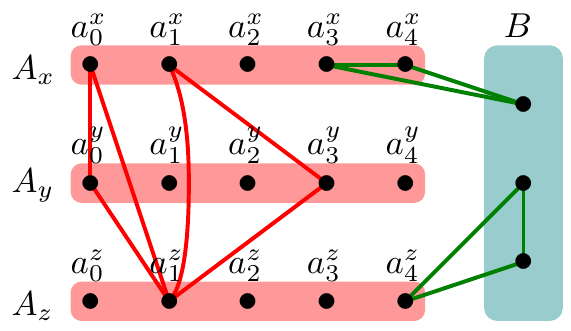}
    \caption{Construction of the hyperedges appearing in $\cC_1$ (green triangles) and in $\cC_T$ (red triangles), where $T=\{x,y,z\}$ is a triple of the Steiner $(p,3,2)$-system.}
    \label{fig:cons}
\end{figure}

\begin{claim}\label{cl:steiner}
Assume that $u\in A_x$ and $v\in A_y$ where $T=\{x,y,z\}$ is a triple in $\cT$. Then $\FC_{\Phi_{\cC_1\cup\cC_2}}(\{u,v\})=V$. 
\end{claim}
\begin{proof}
By Claim~\ref{cl:groups}, it suffices to show that $\FC_{\Phi_{\cC_2}}(\{u,v\})$ contains two elements of the same group. If $u=a^x_i$ and $v=a^y_i$ for some $0\leq i\leq 4$, then $\cC_2$ contains both $\{a^x_i,a^y_i,a^z_{i+1}\}$ and $\{a^x_i,a^y_i,a^z_{i+2}\}$, and we are done. Otherwise, we may assume without loss of generality that $u=a^x_i$ and $v=a^y_j$ where $(j-i \bmod\ 5)\leq 2$. Then $\cC_2$ contains both $\{a^x_i,a^z_i,a^y_j\}$ and $\{a^x_i,a^z_i,a^y_{3+2i-j}\}$. This finishes the proof of the claim.
\end{proof}

By Claims \ref{cl:groups} and \ref{cl:steiner}, the hypergraph $\cD=\cC_1\cup\cC_2$ satisfies $h_\bM=\Phi_{\cD}$ and has total size $n^2/5+O(n)$, concluding the proof of the theorem.
\end{proof}

Finally, we consider the objective $|\cdot|_K$ for the uniform case. 

\begin{theorem} \label{thm:uniform_k}
Let $\bM=(V,\cC)$ be a rank-$r$ uniform matroid on $n\geq r+1$ elements. Then $|\bM|_K= {n\choose r}$.
\end{theorem}
\begin{proof}
Recall that $F_{h_\bM}(X)=V$ for any subset $X\subseteq V$ of size $r$. This implies that each $r$-element subset has to be the body of at least one hyperedge in any representation of $h_\bM$ using circuit clauses, implying $|\bM|_K\geq {n\choose r}$.

For the upper bound, let $v_1,\dots,v_n$ denote the elements of $V$. Take an arbitrary set $X\subseteq V$ of size $r$. Each such set has the form $X=\bigcup_{j=1}^q\{v_{i_j},\dots,v_{i_j+\delta_j}\}$ for some $q\geq 1$, where $i_j+\delta_j+1<i_{j+1}$ for $j\in[q-1]$ and indices are meant in a cyclic order. Define $v_X=v_{i_1+\delta_1+1}$ and
\begin{equation*}
\Phi=\bigwedge_{\substack{X\subseteq V:\\ |X|=r}}X\rightarrow v_X.
\end{equation*}
Then $F_\Phi(X)=X$ if $|X|\leq r-1$ and $F_\Phi(X)=V$ otherwise, therefore $\Phi$ is a representation of $h_\bM$ that uses circular clauses only, concluding the proof.
\end{proof}

%%%%%%%%%%%%%%%%
\section{Conclusions}
\label{sec:conclusions}
%%%%%%%%%%%%%%%%

In the present paper, as a continuation of the work started in \cite{hypergraph_horn}, we considered hypergraph Horn functions associated to families of circuits of a matroid, and introduced the notion of matroid Horn functions. We gave several equivalent characterizations of matroid Horn functions in terms of their canonical and complete CNF representations, and studied minimum representations of matroids and matroid Horn functions with respect to various objectives. 

The proposed subclass opens up new research directions, hence we conclude our paper with listing a few open problems.

\begin{qu}
What is the computational complexity of checking if a given a definite Horn function $h$ represented by a definite Horn CNF $\Psi$ is matroid Horn or not?
\end{qu}

\begin{qu}
Given a rank-$r$ uniform matroid $\bM$, what is the right order of magnitude of $|\bM|_C=b(n,r+1,r)$?
\end{qu}

We conjecture that the answer to this question is in fact the lower bound provided by Theorem~\ref{thm:uniform_c_1}, that is, $\frac{2\cdot n^r}{(2\cdot r+1)\cdot r!}+O(n^{r-1})$.  

\medskip
\paragraph{Acknowledgements} This  work  was  supported  by  the  Research  Institute  for  Mathematical  Sciences,  an  International Joint Usage/Research Center located in Kyoto University. 
The first author was supported by the Lend\"ulet Program of the Hungarian Academy of Sciences -- grant number LP2021-1/2021 and by the Hungarian National Research, Development and Innovation Office -- NKFIH, grant number FK128673.
The third author was supported by JSPS KAKENHI Grant Numbers JP20H05967, JP19K22841, and JP20H00609.

\bibliographystyle{abbrv}
\bibliography{matroidal_horn}

\end{document}